\documentclass[a4paper,10pt,reqno]{amsart}

\usepackage[utf8]{inputenc}
\usepackage[T1]{fontenc}
\usepackage{amsthm}
\usepackage{amsmath}
\usepackage{amssymb}
\usepackage[inline]{enumitem}
\usepackage{comment}
\PassOptionsToPackage{hyphens}{url}\usepackage{hyperref}
\usepackage{fancyhdr}
\usepackage{mathrsfs}
\usepackage{stmaryrd}
\usepackage[normalem]{ulem}
\usepackage{xcolor}
\usepackage{tikz-cd}
\usetikzlibrary{arrows}
\usepackage{cite}
\setlist[description]{%
  itemsep=0.05cm,               
  font={\normalfont\textsc}, 
 leftmargin=\parindent,
 labelindent=\parindent
}

\theoremstyle{definition}

\newtheorem{theorem}{Theorem}[section]
\newtheorem{lemma}[theorem]{Lemma}
\newtheorem{proposition}[theorem]{Proposition}
\newtheorem{corollary}[theorem]{Corollary}

\theoremstyle{definition}
\newtheorem{definition}[theorem]{Definition}
\newtheorem{example}[theorem]{Example}
\newtheorem{remark}[theorem]{Remark}

\newtheorem{problem}[theorem]{Problem}

\newtheorem*{conjecture*}{Conjecture}

\newtheorem*{problem*}{Problem}
\newtheorem*{question*}{Question}

\definecolor{blue-url}{RGB}{0,0,100}
\definecolor{red-url}{RGB}{100,0,0}
\definecolor{green-url}{RGB}{0,100,0}
\definecolor{light-yellow}{RGB}{255,255,128}
\definecolor{light-blue}{RGB}{193,255,255}
\definecolor{light-red}{RGB}{239,83,80}

\hypersetup{
	pdftitle={On the strong isomorphism problem for power semigroups of inverse semigroups},
	pdfauthor={Liu and Tringali},
	pdfmenubar=false,
	pdffitwindow=true,
	pdfstartview=FitH,
	colorlinks=true,
	linkcolor=blue-url,
	citecolor=green-url,
	urlcolor=red-url
}

\renewcommand{\emptyset}{\varnothing}
\renewcommand{\setminus}{\smallsetminus}
\renewcommand{\,}{\kern 0.1em}

\newcommand{\evid}[1]{\textsf{#1}}
\newcommand{\fin}{\mathrm{fin}}

\providecommand{\lLoc}{\textrm{l-Loc}}
\providecommand{\rLoc}{\textrm{r-Loc}}
\providecommand{\Loc}{\textrm{Loc}}

\providecommand{\lFix}{\textrm{l-Fix}}

\providecommand{\EE}{\mathcal{E}}
	
{\newline\vspace{\abovedisplayskip}\hbox to \textwidth\bgroup\hss$\displaystyle}
{$\hss\egroup\vspace{\belowdisplayskip}}

\makeatletter
\DeclareFontFamily{OMX}{MnSymbolE}{}
\DeclareSymbolFont{MnLargeSymbols}{OMX}{MnSymbolE}{m}{n}
\SetSymbolFont{MnLargeSymbols}{bold}{OMX}{MnSymbolE}{b}{n}
\DeclareFontShape{OMX}{MnSymbolE}{m}{n}{
	<-6>  MnSymbolE5
	<6-7>  MnSymbolE6
	<7-8>  MnSymbolE7
	<8-9>  MnSymbolE8
	<9-10> MnSymbolE9
	<10-12> MnSymbolE10
	<12->   MnSymbolE12
}{}
\DeclareFontShape{OMX}{MnSymbolE}{b}{n}{
	<-6>  MnSymbolE-Bold5
	<6-7>  MnSymbolE-Bold6
	<7-8>  MnSymbolE-Bold7
	<8-9>  MnSymbolE-Bold8
	<9-10> MnSymbolE-Bold9
	<10-12> MnSymbolE-Bold10
	<12->   MnSymbolE-Bold12
}{}

\let\llangle\@undefined
\let\rrangle\@undefined
\DeclareMathDelimiter{\llangle}{\mathopen}%
{MnLargeSymbols}{'164}{MnLargeSymbols}{'164}
\DeclareMathDelimiter{\rrangle}{\mathclose}%
{MnLargeSymbols}{'171}{MnLargeSymbols}{'171}
\makeatother

\begin{document}
\title{Kobayashi's problem for power semigroups \\ of inverse semigroups}

\author{Shuolin Liu}
\address{(S.~Liu) School of Mathematical Sciences, Hebei Normal University | Shijiazhuang, Hebei province, 050024 China}
\email{liushuolin77@gmail.com}

\author{Salvatore Tringali}
\address{(S.~Tringali) School of Mathematical Sciences, Hebei Normal University | Shijiazhuang, Hebei province, 050024 China}
\email{salvo.tringali@gmail.com}
\urladdr{https://salvo-tringali.github.io/home/}

%


\begin{abstract}
For a semigroup $S$ (written multiplicatively), let $\mathcal P(S)$ denote the semigroup of all non-empty subsets of $S$ under setwise multiplication. The finitary power semigroup $\mathcal P_{\fin}(S)$ is the subsemigroup of $\mathcal P(S)$ consisting of all non-empty finite subsets of $S$.

Kobayashi's problem asks for which classes $\mathcal O$ of
semigroups every isomorphism $\mathcal P(H) \to \mathcal P(K)$
with $H, K \in \mathcal O$ maps singletons bijectively onto
singletons and hence induces an isomorphism $H \to K$.
We give an affirmative answer for inverse semigroups,
extending previous results for groups, semilattices, and, more generally,
Clifford semigroups. We also establish the corresponding
rigidity theorem for finitary power semigroups.
To this end, we develop a new framework that provides a common approach to both settings.
\end{abstract}

\maketitle

\thispagestyle{empty}

\section{Introduction}
\label{sec:intro}

Let $S$ be a multiplicatively written semigroup (see Section~\ref{sec:preliminaries} for notation and terminology). The non-empty
subsets of $S$ then form a semigroup under the binary operation of \evid{setwise multiplication}
\[
(X, Y) \mapsto XY:=\{xy:x\in X,\, y\in Y\}.
\]
This semigroup is denoted by $\mathcal P(S)$ and called the \evid{large power semigroup} of $S$; other names used in the literature include \evid{global} and
\evid{semigroup of complexes}. The subsemigroup
$\mathcal P_\fin(S)$ of $\mathcal P(S)$ consisting of the non-empty finite subsets of $S$ is the \evid{finitary power semigroup} of $S$. 
Note also that setwise multiplication extends naturally to the full power set of $S$ by taking $X \emptyset = \emptyset X := \emptyset$ for every $X \subseteq S$.
Throughout, however, we exclude the empty set from $\mathcal P(S)$.

Power semigroups were first systematically studied by Tamura and Shafer \cite{Tam-Sha1967}
in the late 1960s. They  serve as an intrinsic algebraic framework for many
questions lying at the intersection of arithmetic
combinatorics and (semi)group theory, while providing an interesting
benchmark for the study of (non-unique) factorization beyond the
classical commutative and cancellative setting \cite{Fa-Tr18, An-Tr18, Tr20(c), Co-Tr-25(a), Rei-2026}.
For further details on these connections, we refer the reader to the survey \cite{Tri-Survey}.

A basic question tracing back to Tamura and Shafer's work \cite{Tam-Sha1967, Shaf-1967} in the late 1960s asks to what extent the structure of a semigroup can be recovered from its large power semigroup, and a number of variants involving other power constructions have been considered over the years \cite[Section 1]{GarSan-Tri-24(a)}. 

The starting observation is that any semigroup isomorphism $\phi \colon H \to K$ induces an isomorphism $\phi^\ast \colon \allowbreak \mathcal P(H) \to \allowbreak \mathcal P(K)$, called the \evid{augmentation} of $\phi$, via the map $X \mapsto \allowbreak \phi[X]$. 
We are thus led to the following question, often referred to as the \evid{Tamura--Shafer problem} for power semigroups.

\begin{problem}
\label{ques:tamura-shafer-problem}
Determine whether a given class $\mathcal O$ of semigroups is \evid{globally determined}: for all
$H, K \in \mathcal O$, does the existence of an isomorphism $\mathcal P(H) \to \mathcal P(K)$ imply the existence of an isomorphism $H \to K$?
\end{problem}
The problem has a negative answer for the class of all semigroups \cite{Mog1973} and remains open for finite semigroups, despite claims to the contrary based on announcements made by Tamura in \cite{Tamu-1987}. For a detailed discussion of this issue, see
\begin{center}
\url{https://mathoverflow.net/questions/508548/}
\end{center}

On the positive side, Tamura proved in
\cite[Theorems~5.9 and 6.8]{Tam-1986} that
Problem~\ref{ques:tamura-shafer-problem} has an affirmative answer for completely simple and completely $0$-simple semigroups. More recently, Yu and Zhao established in
\cite[Theorem~5.4]{Yu-Zhao-2025} that the same is true for
\evid{completely regular semigroups}, that is, semigroups
each of whose elements lies in a subgroup.
All these results rely on the Generalized Continuum Hypothesis. In many other cases where the problem has an affirmative answer, a stronger rigidity property holds, motivating the following:

\begin{problem}
\label{ques:kobayashi-problem}
Determine whether a given class $\mathcal O$ of semigroups is \evid{strongly globally determined}: for all
$H, K \in \mathcal O$, does every isomorphism $\mathcal P(H) \to \mathcal P(K)$ map singletons to singletons?
\end{problem}

This is sometimes referred to as the \evid{Kobayashi problem} for power semigroups, as it was first explicitly formulated and solved in the affirmative by Kobayashi for the class of semilattices \cite[p.~218]{Koba-1984}, that is, commutative semigroups all of whose elements are idempotent. The result was generalized by Gan and Zhao \cite[Theorem 4.7]{Gan-Zhao-2014}, who proved that Problem~\ref{ques:kobayashi-problem} has a positive answer for Clifford semigroups, that is, completely regular semigroups in which every idempotent is central. 

Much less is known about the following finitary analogues of Problems~\ref{ques:tamura-shafer-problem} and \ref{ques:kobayashi-problem}, apart from the cases where the semigroups under consideration are finite and hence $
\mathcal P(S) = \mathcal P_{\mathrm{fin}}(S)$.

\begin{problem}
\label{ques:finitary-problems}
Given a class $\mathcal O$ of semigroups, prove or disprove the following statements:

\begin{enumerate}[label=\textup{(\roman*)}, leftmargin=*, itemsep=0pt]
\item\label{ques:finitary-problems(i)} For all $H, K\in\mathcal O$, an isomorphism $\mathcal P_\fin(H) \to \mathcal P_\fin(K)$ implies an isomorphism $H \to K$.

\item\label{ques:finitary-problems(ii)} For all $H, K\in\mathcal O$, every isomorphism $\mathcal P_\fin(H) \to \mathcal P_\fin(K)$
maps singletons to singletons.
\end{enumerate}
\end{problem}

The augmentation $\phi^\ast$ of an isomorphism $\phi \colon H \to K$ restricts to an isomorphism $\mathcal P_\fin(H) \to \mathcal P_\fin(K)$, since $\phi^\ast(X)$ is finite and non-empty for every $X \in \mathcal P_\fin(H)$. 
A non-trivial class for which both Problems~\ref{ques:kobayashi-problem} and \ref{ques:finitary-problems}\ref{ques:finitary-problems(ii)} have an affirmative answer is that of commutative cancellative semigroups \cite[Theorem 1]{Tri-2024(a)}, while it remains open whether the same conclusion carries over to arbitrary cancellative semigroups.

Interest in the Kobayashi problem and its finitary analogue is partly
motivated by the study of automorphisms of power semigroups. This aspect
of the theory was first considered by Byrd et al.~in a series of papers from the 1970s and 1980s \cite{Byr-Llo-Ped-Ste-1984, Byr-Llo-Ste-1982,
Byr-Llo-Men-Tel-1978, Byr-Llo-Ped-Ste-1977}
and has attracted renewed attention in recent years
\cite{Tri-Yan2025(b), Rago26(b), Tri-Wen-2026(b),
Tri-Wen-2026(a), Wong2025, Wong2026}. Related questions have been investigated in \cite{Bie-Ger-22, Tri-Yan-23(a), Tri-Yan2026(a), GarSan-Tri-24(a), Li-Tri-2024, Rago26, Rago26(c), Liu-Tri-2026, Bho-Tri-2026}. 

The present work adds to these lines of research. More precisely, recall
from \cite[Chapter~5]{Ho95} that an \evid{inverse semigroup} is a
regular semigroup whose idempotents form a semilattice, where a
semigroup $S$ is \evid{regular} if, for each $x\in S$, there exists
$y\in S$ such that $x=xyx$. We prove in
Corollaries~\ref{cor:kobayashi-theorem}
and~\ref{cor:finitary-kobayashi-for-inverse-semigroups} that
Problems~\ref{ques:kobayashi-problem}
and~\ref{ques:finitary-problems}\ref{ques:finitary-problems(ii)} both have a positive answer for
the class of inverse semigroups.

The aforementioned result of Yu and Zhao on completely regular
semigroups \cite[Theorem~5.4]{Yu-Zhao-2025} is not implied by ours,
since neither the class of inverse semigroups nor that of completely
regular semigroups contains the other. However, the intersection of
these two classes is precisely the class of Clifford semigroups, and
hence our result strictly generalizes
\cite[Theorem~4.7]{Gan-Zhao-2014}.

Perhaps more importantly, our approach appears to be new and readily adaptable to other settings. In particular, it exploits the properties
of \textit{localizers}, a tool that we introduce and investigate in
Section~\ref{sec:localizers}.

\section{Preliminaries}
\label{sec:preliminaries}

We write $|V|$ for the cardinality and $2^V$ for the power set of a set $V$. Given a function $\phi \colon A \to \allowbreak B$ and a subset $X$ of $A$, we denote by $\phi[X] := \allowbreak \{\phi(x): \allowbreak x \in X\}$ the direct image of $X$ under $\phi$. 

We refer the reader to Howie’s monograph \cite{Ho95} for the basic aspects of semigroup theory relevant to this work. Unless otherwise specified, all morphisms are understood to be semigroup homomorphisms and all semigroups are written multiplicatively. If $S$ is a semigroup, we denote by $\EE(S) := \{e \in S: e^2 = e\}$ its set of idempotents, and we say that a subfamily $\mathcal D$ of the large power semigroup $\mathcal P(S)$ of $S$ is 
\begin{itemize}
\item \evid{hereditary} if $Y \in \mathcal D$ for every \textit{non-empty} subset $Y$ of a set $X \in \mathcal D$;
\item \evid{union-closed} if $X \cup Y \in \mathcal D$ for all $X, Y \in \mathcal D$.
\end{itemize}
In particular, it is obvious that $\mathcal P(S)$ and $\mathcal P_\fin(S)$ are both hereditary and union-closed. 

Finally, we shall use without further comment that, for each element $a$ in an inverse semigroup $S$, there is a unique $b \in S$, called the \evid{inverse} of $a$ (in $S$), such that $aba = a$ and $bab = b$  \cite[Theorem 5.1.1]{Ho95}. 

Additional notation and terminology, if not explained when first used, are standard or should be clear from the context.

\section{Localizers}
\label{sec:localizers}

Below we introduce the notion of a localizer, which is central to the present work and, we believe, will prove useful well beyond its scope. The section is largely devoted to studying the properties of these objects. The term ``localizer'' is meant to emphasize that a certain existence and uniqueness condition is imposed within a prescribed subset of a semigroup $S$, rather than across all of $S$.

\begin{definition}
\label{def:localizers}
Let $S$ be a semigroup and $\vartheta$ be a function from $S$ to its power set $2^S$, and fix an element $a \in S$. The \evid{left} \evid{$\vartheta$-localizer}
of $a$ (in $S$), denoted by $
\lLoc_S(a, \vartheta)$, is the set  of all $b \in S$ such that $x = b$ is the unique solution $x \in \vartheta(b)$ to the equation
$ax = b$ (note that this condition implies $b \in \vartheta(b) \ne \emptyset$). 

The \evid{right $\vartheta$-localizer} of $a$ (in $S$), denoted by $\rLoc_S(a,\vartheta)$, is defined symmetrically, with right
multiplication by $a$ in place of left multiplication. We call the intersection $
\lLoc_S(a,\vartheta) \cap \rLoc_S(a,\vartheta)$
the \evid{two-sided $\vartheta$-localizer} of $a$ (in $S$), and we denote it by $\Loc_S(a,\vartheta)$. 
\end{definition}

We will omit 
the subscript ``$S\,$'' from the above notation when this is unlikely to cause any ambiguity. The following special cases are of particular interest.

\begin{itemize}
\item Suppose that $\vartheta(b) = \{b\}$ for every $b \in S$.
Then the left $\vartheta$-localizer of $a$ is precisely the
fixed-point set of left multiplication by $a$. In other words, we have
\[
\lLoc_S(a,\vartheta)
=
\lFix_S(a)
:=
\{b \in S : ab = b\}.
\]
Similarly, the right $\vartheta$-localizer is the fixed-point
set of right multiplication by $a$, while the two-sided $\vartheta$-localizer consists of all $b \in S$ such that $ab = ba = b$.

\item If $\vartheta$ is constant with value $\Gamma \subseteq S$, we write $\lLoc_S(a, \Gamma)$ in place of $\lLoc_S(a, \vartheta)$, 
and use analogous notation for the right and two-sided localizers. More explicitly,
\[
\lLoc_S(a,\Gamma)
=
\{b \in \lFix_S(a) \cap \Gamma : ax \ne b \text{ for every } x \in \Gamma \text{ with }x \ne b\}.
\]
In particular, we denote $\lLoc_S(a,S)$ by $\lLoc_S(a)$. Therefore, $b \in \lLoc_S(a)$ if and only if $b$ is the unique solution $x \in S$ to the equation $ax = b$. The right localizer
$\rLoc_S(a)$ and the two-sided localizer $\Loc_S(a)$ are defined similarly.
\end{itemize}

These notions are compatible with isomorphisms, as we make precise in our first proposition.

\begin{proposition}
\label{prop:isos-preserve-localizers}
Let $\phi \colon S \to T$ be a semigroup isomorphism and $\vartheta$ be a function $S \to 2^S$, and denote by $\vartheta^\phi$ the map $T \to 2^T \colon y \mapsto \phi[\vartheta \circ \phi^{-1}(y)]$. Then
\begin{equation}
\label{equ:localizers-are-preserved}
\phi[\lLoc_S(a, \vartheta)]
=
\lLoc_T(\phi(a), \vartheta^\phi),
\qquad\text{for every }
a \in S.
\end{equation}
In addition, analogous identities hold for right and two-sided localizers.
\end{proposition}

\begin{proof}
We prove the statement only for left localizers, as the other cases are analogous. 

Fix $a, b \in S$. Since $\phi$ is an isomorphism, the restriction of $\phi$ to $\vartheta(b)$ induces a bijection onto the set $\phi[\vartheta(b)] = \vartheta^\phi(\phi(b))$. Moreover, for every $x \in S$ we have
\[
ax = b
\quad\Longleftrightarrow\quad
\phi(a)\phi(x) = \phi(b). 
\]
Thus $b$ is the unique solution $x \in \vartheta(b)$ to the equation $ax = b$ if and only if $\phi(b)$ is the unique solution $y \in \vartheta^\phi(\phi(b))$ to $\phi(a)y = \phi(b)$. Consequently,
\[
b \in \lLoc_S(a,\vartheta)
\quad\Longleftrightarrow\quad
\phi(b) \in \lLoc_T(\phi(a),\vartheta^\phi),
\]
and the desired equality follows from the surjectivity of $\phi$.
\end{proof}

Before proceeding, it is perhaps worth emphasizing that, while many of the results in this and subsequent sections are formulated for left localizers, their right-handed analogues follow by symmetry, and their two-sided versions are obtained by combining the corresponding one-sided statements.

\begin{proposition}
\label{prop:localizer-and-inclusion}
Let $S$ be a semigroup, $\mathcal D$ be a subsemigroup of $\mathcal P(S)$, and $\vartheta$ be a function $\mathcal D \to 2^{\mathcal D}$ each of whose values is a hereditary family. If $A \in \EE(\mathcal D)$ and $X \in \lLoc_\mathcal{D}(A, \vartheta)$, then 
\[
BY = \allowbreak Y,
\qquad\text{for all non-empty sets } 
B \subseteq A \text{ and } Y \subseteq X.
\]
\end{proposition}

\begin{proof}
Fix a non-empty $B \subseteq A$. It suffices to prove that $Bx = \{x\}$ for every $x \in X$, since then
\[
BY = \bigcup_{y \in Y} By
= \bigcup_{y \in Y} \{y\}
= Y,
\qquad\text{for each non-empty } Y \subseteq X.
\]

To start with, $X \in \lLoc_{\mathcal D}(A, \vartheta)$ yields
$X \in \vartheta(X)$ and $AX = X$. We claim that
\begin{equation}
\label{lem:localizer-and-inclusion:eq(2)}
x \in Ax, \qquad \text{for all } x \in X.
\end{equation}
Indeed, let $x \in X$. Considering that $Ax \subseteq AX = X$, the set
\begin{equation}
\label{lem:localizer-and-inclusion:eq(3)}
Y_x := (X \setminus \{x\}) \cup Ax
\end{equation}
is a non-empty subset of $X$ and hence belongs to $\vartheta(X)$, because $X$ is in $\vartheta(X)$ and the latter is, by hypothesis, a hereditary subfamily of $\mathcal D$. Moreover, $A$ being an idempotent of $\mathcal D$ implies
\[
X = AX \supseteq AY_x
=
A(X \setminus \{x\}) \cup A^2x
=
A(X \setminus \{x\})\cup Ax \supseteq (AX \setminus Ax) \cup Ax =
AX =
X,
\]
that is, $AY_x = X$. Thus, (the uniqueness property inherent to) the definition of
$\lLoc_{\mathcal D}(A,\vartheta)$ forces $Y_x = \allowbreak X$, which,
by Eq.~\eqref{lem:localizer-and-inclusion:eq(3)}, implies in turn that
$x \in Ax$, as claimed.

Now suppose, by way of contradiction, that there exist $x \in X$ and $y \in Ax$ with
$y \ne x$. Since $Ax \subseteq \allowbreak AX = \allowbreak X$, it is then clear that $Z := X \setminus \{y\}$
is a proper subset of $X$ containing $x$. By the hereditariness of $\vartheta(X)$, it follows that $X \ne Z \in \vartheta(X)$. On the other hand, we are guaranteed by Eq.~\eqref{lem:localizer-and-inclusion:eq(2)} that 
\[
X \setminus \{y\} = Z \subseteq AZ \subseteq AX = X
\]
while $x \in Z$ and $y \in Ax$ give $y \in AZ$. This shows that
\[
X = Z \cup \{y\} \subseteq AZ \subseteq X.
\]
That is, we have $AZ = X$, contradicting the definition of $\lLoc_\mathcal{D}(A, \vartheta)$ and proving that $Ax = \{x\}$. 

Accordingly, for every $x \in X$, we find that $\emptyset \ne Bx \subseteq Ax = \{x\}$ and hence $Bx = \{x\}$, as desired.
\end{proof}

\begin{remark}
\label{remark:two-gauge-functions}
Let $S$ be a semigroup, and let $\vartheta$ be a function from $S$ to its power set. Given $a, x \in S$, it follows immediately from
Definition~\ref{def:localizers} that 
\begin{equation}
\label{remark:two-gauge-functions:eq(1)}
x \in \lLoc_S(a,\vartheta) 
\quad\Longleftrightarrow\quad
x \in \lLoc_S(a,\vartheta(x)).
\end{equation}
Moreover, if $\lambda$ is another function $S \to 2^S$, then
\begin{equation}
\label{remark:two-gauge-functions:eq(2)}
x \in \vartheta(x) \subseteq \lambda(x) \text{ and } x \in \lLoc_S(a,\lambda)
\quad\Longrightarrow\quad
x \in \lLoc_S(a,\vartheta).
\end{equation}

Together with Proposition~\ref{prop:localizer-and-inclusion}, these observations
provide a useful strategy for proving
that, under suitable hypotheses, an isomorphism between power semigroups maps one-element idempotents to one-element idempotents. We will see an application to inverse semigroups in the proof of
Theorem~\ref{thm:finitary-kobayashi-theorem}. 
\end{remark}

Now we strengthen Proposition~\ref{prop:localizer-and-inclusion} by giving sufficient conditions for localizers to be hereditary.

\begin{proposition}
\label{lem:localizer-and-inclusion-hereditary}
Let $S$ be a semigroup, $\mathcal D$ be a subsemigroup of $\mathcal P(S)$, and $\vartheta$ be a function $\mathcal D \to 2^\mathcal{D}$ each of whose values is a hereditary, union-closed family. Assume further that $X \in \vartheta(X)$ for all $X \in \mathcal D$, and that $\vartheta$ is monotone under inclusion (in the sense that $\vartheta(Y) \subseteq \vartheta(X)$ whenever $X, Y \in \mathcal D$ and $Y \subseteq X$). Then $\lLoc_{\mathcal D}(A, \vartheta)$ is a hereditary subfamily of $\mathcal D$ for every $A \in \EE(\mathcal D)$.
\end{proposition}

\begin{proof}
Let $X \in \lLoc_{\mathcal D}(A,\vartheta)$, and fix a
non-empty set $Y \subseteq X$. Since $\vartheta(X)$ contains $X$ (by the definition of localizers) and $\vartheta(X)$ is hereditary (by the hypotheses on $\vartheta$), we have $Y \in \vartheta(X) \subseteq \mathcal D$. Therefore, $\vartheta(Y)$ contains $Y$ (again by hypothesis), while Proposition~\ref{prop:localizer-and-inclusion} yields $AY = Y$. It remains to see that 
\[
AZ \ne Y,
\qquad\text{for every }
Z \in \vartheta(Y) \text{ with } Z \ne Y.
\]

Assume to the contrary that there exists
$Z \in \vartheta(Y)$ such that $AZ = Y \ne Z$.
If $Z \subseteq X$, then
$AZ = Z$ by Proposition~\ref{prop:localizer-and-inclusion}, and hence $Z = Y$, which is absurd. Thus $Z \nsubseteq X$, and hence 
\begin{equation}
\label{lem:localizer-and-inclusion-hereditary:eq(1)}
A(X \cup Z) = AX \cup AZ = X \cup Y = X \ne X \cup Z.
\end{equation}

Now, $Y \subseteq X$ implies $\vartheta(Y) \subseteq \vartheta(X)$ and hence $Z \in \vartheta(X)$, by the monotonicity of $\vartheta$. Accordingly, $\vartheta(X)$ being union-closed (yet again by hypothesis) yields $X \cup Z \in \vartheta(X)$. Together with Eq.~\eqref{lem:localizer-and-inclusion-hereditary:eq(1)}, this is however in contradiction with the assumption that
$X \in \lLoc_{\mathcal D}(A, \vartheta)$, thereby completing the proof.
\end{proof}

\begin{corollary}
\label{cor:localizer-and-inclusion-hereditary}
Let $S$ be a semigroup and $\mathcal D$ be a hereditary, union-closed subsemigroup of $\mathcal P(S)$. Then
$\lLoc_{\mathcal D}(A)$ is a hereditary subfamily of $\mathcal D$ for every $A \in \EE(\mathcal D)$.
\end{corollary}

\begin{proof}
Straightforward from
Proposition~\ref{lem:localizer-and-inclusion-hereditary}
and our definitions (the details are left to the reader).
\end{proof}

The following examples show that the hypotheses of Proposition~\ref{lem:localizer-and-inclusion-hereditary} cannot be dropped altogether.

\begin{example}
Let $\Omega := \{0, 1, 2\} \subseteq \mathbb N$, and denote by $T_\Omega$ the full transformation monoid on $\Omega$, that is, the monoid of all self-maps on $\Omega$ under composition. Define
\[
\alpha :=
\begin{pmatrix}
0 & 1 & 2\\
0 & 1 & 0
\end{pmatrix}
\in T_\Omega,
\]
and for each $i \in \Omega$ let $c_i$ be the constant map on $\Omega$ with value $i$. Set
\[
S := \{\alpha, c_0, c_1, c_2\}
\quad\text{and}\quad
\mathcal D :=
\bigl\{
\{\alpha\}, \{c_0\}, \{c_1\}, \{c_2\}, \{c_0, c_1\}
\bigr\}.
\]
Note that $S$ is a subsemigroup of $T_\Omega$ and $\mathcal D$ is a hereditary subsemigroup of $\mathcal P(S)$; in particular, 
\begin{equation}
\label{equ:full-transformation-monoid}
\alpha \circ c_0 = \alpha \circ c_2 = c_0
\quad\text{and}\quad
\alpha \circ c_1 = c_1. 
\end{equation}
However, $\mathcal D$ is not union-closed, since
\[
\{c_0, c_1\} \cup \{c_2\}
=
\{c_0, c_1, c_2\}
\notin\mathcal D.
\]

Now, let $A := \{\alpha\}$ and $X := \{c_0, c_1\}$. It is immediate from Eq.~\eqref{equ:full-transformation-monoid} that $AX = X$; moreover,
\[
AZ \ne X,
\qquad\text{for every } Z \in \mathcal D \text{ with } Z \ne X.
\]
Therefore, $X \in \lLoc_{\mathcal D}(A)$. On the other hand, letting $Y := \{c_0\}$ yields
\[
AY=Y
\quad\text{and}\quad
A\{c_2\} = \{c_0\} = Y.
\]
Since $\{c_2\} \ne Y \subsetneq X$, this shows that
$Y \notin \lLoc_{\mathcal D}(A)$, and hence
$\lLoc_{\mathcal D}(A)$ is not hereditary. 
It is thus clear that union-closedness plays a somewhat critical role in ensuring that localizers are hereditary.
\end{example}

\begin{example}
Let $S$ be the semilattice obtained by endowing $\{0, 1, 2\} \subseteq \mathbb N$ with the binary operation $(a, b) \mapsto \min(a, b)$ relative to the usual ordering of the non-negative integers, and let $\mathcal D := \mathcal P(S)$.

Each non-empty subset $Y$ of $S$ satisfies $Y^2 = Y$; in particular, $Y$ is a subsemigroup of $S$, and hence we may consider its large power semigroup $\mathcal P(Y)$.
Accordingly, let $\vartheta$ be the function
\[
\mathcal D\to 2^{\mathcal D} \colon Y \mapsto 
\begin{cases}
\mathcal P(\{1, 2\}), & \text{if } Y = \{1\},\\
\mathcal P(Y), & \text{otherwise}.
\end{cases}
\]
Each value of $\vartheta$ is a hereditary, union-closed
family, and $Y \in \vartheta(Y)$ for all $Y \in \mathcal D$.

Now put $A := \{1\}$ and $X := \{0,1\}$. Then
$A \in \EE(\mathcal D)$, and $AY = Y$ for every non-empty $Y \subseteq X$. Since $\vartheta(X) = \mathcal P(X)$,
this implies $X \in \lLoc_{\mathcal D}(A,\vartheta)$. On the other hand, $A$ is a non-empty subset of $X$, but 
\[
\{2\} \in \vartheta(A)
\quad\text{and}\quad
A\{2\} = A \ne \{2\}.
\]
It follows that $A \notin \lLoc_{\mathcal D}(A,\vartheta)$, and hence $\lLoc_{\mathcal D}(A,\vartheta)$ is not hereditary. In fact, 
Proposition~\ref{lem:localizer-and-inclusion-hereditary}
does not apply here, because $\vartheta$ is not monotone
under inclusion: $A \subseteq X$, but
$\{2\} \in \vartheta(A) \setminus \vartheta(X)$.
\end{example}

For use in the next results, we recall that a \evid{right-zero semigroup} is a semigroup $S$ such that $xy = y$ for all $x, y \in S$. Left-zero semigroups are defined symmetrically. 
 
\begin{proposition}
\label{prop:key-property-of-left-localizers}
Let $S$, $\mathcal D$, and $\vartheta$ be  as in Proposition \ref{lem:localizer-and-inclusion-hereditary}. If $A \in \EE(\mathcal D)$ and $E \in \mathcal D$ are such that
\begin{equation}
\label{prop:key-property-of-left-localizers:eq(1)}
E^2 \subseteq E
\quad\text{and}\quad
E \in \lLoc_{\mathcal D}(A, \vartheta) \subseteq E\mathcal D,
\end{equation}
then $E$ is a right-zero subsemigroup of $S$, and $E\mathcal D = \{e\}\mathcal D$ for every $e \in E$. 
In particular, $E \subseteq \EE(S)$.
\end{proposition}

\begin{proof}
Since $A \in \mathcal D$ is idempotent and $\vartheta$ is a monotone function $\mathcal D \to 2^{\mathcal D}$ such that $\vartheta(X)$ is a hereditary, union-closed family containing $X$ for every $X \in \mathcal D$, we gather from Proposition~\ref{lem:localizer-and-inclusion-hereditary} that $\lLoc_{\mathcal D}(A, \vartheta)$ is itself a hereditary subfamily of $\mathcal D$. Therefore, Eq.~\eqref{prop:key-property-of-left-localizers:eq(1)} yields
\begin{equation}
\label{prop:key-property-of-left-localizers(2)}
\{y\} \in \lLoc_{\mathcal D}(A, \vartheta) \subseteq E\mathcal D,
\qquad\text{for all } y \in E.
\end{equation}

In turn, this implies that, for each $y \in E$, there exists $X_y \in \mathcal D$ such that $\{y\} = E X_y$, and hence
\[
\emptyset \ne Ey = E(EX_y) = E^2X_y \subseteq EX_y = \{y\},
\]
where we used that $E^2 \subseteq E$ (by hypothesis). In other words, we have
\begin{equation}
\label{equ:right-zero-subsgrp}
xy = y,
\qquad\text{for all } x, y \in E.
\end{equation}
So, $E$ is a right-zero subsemigroup of $S$, and each of its elements is an idempotent of $S$.

Now fix $e \in E$. By
Eq.~\eqref{equ:right-zero-subsgrp}, we have $eE = E$ and $Ee = \{e\}$. Given any $X \in \mathcal D$, it follows that 
\[
EX = (eE)X = \allowbreak e(EX) \in \{e\} \mathcal D,
\]
and thus $E\mathcal D \subseteq \{e\}\mathcal D$.
Conversely, we know from Eq.~\eqref{prop:key-property-of-left-localizers(2)} that $\{e\} \in \mathcal D$. Hence, if $Y \in \allowbreak \mathcal D$, then 
\[
eY \in \mathcal D
\quad\text{and}\quad
eY = (Ee)Y = E(eY) \in E\mathcal D,
\]
which proves that $\{e\}\mathcal D\subseteq E\mathcal D$ and completes the proof.
\end{proof}

\begin{corollary}
\label{cor:key-property-of-left-localizers}
Let $S$ be a semigroup and $\mathcal D$ be a hereditary, union-closed subsemigroup of $\mathcal P(S)$. Assume that there exist $E \in \mathcal D$ and $A \in \mathcal E(\mathcal D)$ such that $
E^2 \subseteq E$ and $
E \in \lLoc_{\mathcal D}(A) \subseteq E\mathcal D$.
Then $E$ is a right-zero sub\-semigroup of $S$, and $E\mathcal D = \{e\} \mathcal D$ for every $e \in E$. 
In particular, $E \subseteq \EE(S)$.
\end{corollary}

\begin{proof}
Straightforward from Proposition~\ref{prop:key-property-of-left-localizers} and our definitions (we leave the details to the reader).
\end{proof}

The following two-sided analogue of Proposition \ref{prop:key-property-of-left-localizers} allows for stronger conclusions.

\begin{proposition}
\label{prop:key-property-of-localizers}
Let $S$, $\mathcal D$, and $\vartheta$ be as in Proposition \ref{lem:localizer-and-inclusion-hereditary}. If $A \in \mathcal E(\mathcal D)$ and $E \in \mathcal D$ are such that
\begin{equation}
\label{prop:key-property-of-localizers:eq(1)}
E^2 \subseteq E
\quad\text{and}\quad
E \in \Loc_{\mathcal D}(A, \vartheta) \subseteq \allowbreak E \mathcal D E,
\end{equation}
then $E = \{e\}$ for some $e \in \EE(S)$.
\end{proposition}

\begin{proof}
By definition, we have $\Loc_\mathcal{D}(A, \vartheta) = \lLoc_\mathcal{D}(A, \vartheta) \cap \rLoc_\mathcal{D}(A, \vartheta)$. Since any intersection of he\-red\-i\-ta\-ry families is he\-red\-i\-tary, it follows from Proposition~\ref{lem:localizer-and-inclusion-hereditary}, its left-right dual, and Eq.~\eqref{prop:key-property-of-localizers:eq(1)} that
\[
\{y\} \in \Loc_{\mathcal D}(A, \vartheta) \subseteq E\mathcal D E,
\qquad\text{for all } y \in E.
\]
Consequently, for each $y \in E$, there is a set $X_y \in \mathcal D$ such that $\{y\} = E X_y E$, and hence
\[
\emptyset \ne Ey = E(EX_yE) \subseteq E X_y E = \{y\},
\]
where we used that $E^2 \subseteq E$ (by hypothesis).
So $E$ is a right-zero semigroup, and a symmetric argument guarantees that it is also left-zero. However, a non-empty semigroup is both right-zero and left-zero if and only if it is trivial. Therefore, $E = \{e\}$ for some $e \in \EE(S)$.
\end{proof}

\begin{corollary}
\label{cor:key-property-of-localizers}
Let $S$ be a semigroup and $\mathcal D$ be a hereditary, union-closed subsemigroup of $\mathcal P(S)$. Fix $A \in \mathcal E(\mathcal D)$ and $E \in \mathcal D$. If $
E^2 \subseteq E$ and $
E \in \Loc_{\mathcal D}(A) \subseteq E\mathcal D E$, then $E = \{e\}$ for some $e \in \EE(S)$.
\end{corollary}

\begin{proof}
This follows directly from Proposition \ref{prop:key-property-of-localizers} and our definitions (we omit further details).
\end{proof}

In particular, Corollaries~\ref{cor:localizer-and-inclusion-hereditary}, \ref{cor:key-property-of-left-localizers}, and \ref{cor:key-property-of-localizers} apply when $\mathcal D$ is either $\mathcal P(S)$ or $\mathcal P_\fin(S)$. Propositions~\ref{lem:localizer-and-inclusion-hereditary}, \ref{prop:key-property-of-left-localizers}, and \ref{prop:key-property-of-localizers} also apply to the same structures when $\vartheta(X) = 2^X \setminus \{\emptyset\}$ for every $X \in \mathcal D$ (the empty set must be excluded, since
$\emptyset \notin \mathcal D$). 
These results provide another effective approach to proving that an isomorphism between power semigroups maps one-element idempotents to one-element idempotents (cf.~Remark~\ref{remark:two-gauge-functions}), as we will see in the proof of Theorem~\ref{thm:kobayashi-theorem}.

The final lemma in this section, though not needed for the paper's main results, establishes a uniqueness property that may be regarded as a natural complement to Proposition~\ref{prop:key-property-of-localizers}.

\begin{lemma}
\label{lem:uniqueness-of-localizing-idempotent}
Let $S$ be a semigroup, $\mathcal D$ be a subsemigroup of
$\mathcal P(S)$, and $\vartheta$ be a function $\mathcal D \to 2^{\mathcal D}$. For every $A \in \mathcal E(\mathcal D)$, there
exists at most one idempotent $e \in \mathcal E(S)$ such that
\begin{equation}
\label{lem:uniqueness-of-localizing-idempotent:eq(1)}
\{e\}\in\Loc_{\mathcal D}(A, \vartheta)
   \subseteq \{e\}\mathcal D\{e\}.
\end{equation}
\end{lemma}

\begin{proof}
Let $u, v \in\EE(S)$, and suppose that Eq.~\eqref{lem:uniqueness-of-localizing-idempotent:eq(1)} is satisfied both for $e = u$ and for $e = v$. Then
\[
\{v\}\in\Loc_{\mathcal D}(A, \vartheta)
   \subseteq\{u\}\mathcal D\{u\}
\quad\text{and}\quad
\{u\}\in\Loc_{\mathcal D}(A, \vartheta)
   \subseteq\{v\}\mathcal D\{v\},
\]
and hence $\{v\} = uXu$ and $\{u\} = vYv$ for some $X, Y \in \mathcal D$. By the idempotency of $u$, it follows that
\[
\{uvu\} =\{u\}\{v\}\{u\}
 =\{u\}(uXu)\{u\}
 =\{u\}X\{u\}
 =\{v\},
\]
that is, $uvu = v$. Symmetrically, $vuv=u$. Thus, 
\[
u=u^2=u(vuv)=(uvu)v=v^2=v, 
\]
which is enough to finish the proof.
\end{proof} 

\section{Kobayashi's problem for inverse semigroups: the infinitary case}
\label{sec:inverse-sgrps}

In this section, we give an affirmative answer to Problem~\ref{ques:kobayashi-problem} for the class of inverse semigroups (see Theorem~\ref{thm:kobayashi-theorem}). To this end, we will need a few auxiliary results to supplement
those of Section~\ref{sec:localizers}.

\begin{lemma}
\label{lem:localizers-in-power-semigroups-of-semilattices}
Let $S$ be an inverse semigroup. For every $e \in \EE(S)$, there exists $A \in \EE(\mathcal P(S))$ such that
\begin{equation}
\label{lem:localizers-in-power-semigroups-of-semilattices:eq(0)}
\{e\} \in\Loc_{\mathcal P(S)}(A)
   \subseteq \{e\} \mathcal P(S) \{e\}.
\end{equation}
\end{lemma}

\begin{proof}
Fix $e \in \EE(S)$. The idempotents of an inverse semigroup commute \cite[Theorem 5.1.1]{Ho95}. Hence
\[
(ab)^2 = abab = a^2 b^2 = ab, 
\qquad\text{for all }
a, b \in \EE(S). 
\]
Consequently, $\EE(S)$ is a semilattice under the operation inherited from $S$, with the result that
\begin{equation}
\label{lem:localizers-in-power-semigroups-of-semilattices:eq(2)}
X = \{x^2: x \in X\} \subseteq X^2,
\qquad\text{for every } X \subseteq \EE(S).
\end{equation}

Set $A := \{a \in \EE(S): ea = e\}$. Obviously, $e \in A$. Moreover, 
\[
e(ab) = (ea)b = eb = e,
\qquad\text{for all }
a, b \in A.
\]
Thus $A^2\subseteq A$, while
Eq.~\eqref{lem:localizers-in-power-semigroups-of-semilattices:eq(2)}
gives $A\subseteq A^2$. Therefore, $A^2=A$, and hence $A \in \EE(\mathcal P(S))$.

We claim that $\{e\} \in \Loc_{\mathcal P(S)}(A)$. By definition, this amounts to showing that $\{e\}$ is the only solution $Y \in \allowbreak \mathcal P(S)$ of each of the equations $AY = \{e\}$ and $YA = \{e\}$. Below we shall focus on the first of these two equations, as the second can be treated similarly.

To begin, it is clear from the above that $Ae = \{e\}$. Suppose, on the other hand, that $AY = \{e\}$ for some $Y \in \mathcal P(S)$; our task is to prove that $Y = \{e\}$. Since $e \in A$, we have 
\begin{equation}
\label{lem:localizers-in-power-semigroups-of-semilattices:eq(3)}
\emptyset \ne eY \subseteq AY = \{e\}.
\end{equation}
Now let $y \in Y$. It is a standard fact that the product of an element $a \in S$ by its inverse is idempotent
\cite[Theorem~2.3.4(1)]{Ho95}; in particular,
$f_y := yy^{-1} \in \EE(S)$. In addition, every idempotent is its own inverse in
$S$, and \cite[Proposition~5.1.2(1)]{Ho95} gives $
(ab)^{-1} = b^{-1}a^{-1}$ for every $a, b \in S$.
Since $ey = e$ by Eq.~\eqref{lem:localizers-in-power-semigroups-of-semilattices:eq(3)} and, as noted earlier, $uv = vu$ for all $u, v \in \EE(S)$, we obtain 
\[
e = ee^{-1} = (ey)(ey)^{-1} = (ey)(y^{-1}e^{-1})
  = ef_ye = e^2 f_y = ef_y.
\]
Thus $f_y \in A$, and consequently
\[
y = yy^{-1}y = f_yy \in AY = \{e\}.
\]
It follows, as desired, that $Y = \{e\}$, because $y$ was an arbitrary element in $Y$. 

A left-right dual argument likewise shows that $YA = \{e\}$ yields $Y = \{e\}$, so it remains to check that 
\[
\Loc_{\mathcal P(S)}(A) \subseteq \{e\} \mathcal P(S) \{e\}. 
\]
To this end, let $X \in \Loc_{\mathcal P(S)}(A)$. Since $e \in A$, we are guaranteed by Proposition~\ref{prop:localizer-and-inclusion} and its left-right dual that $X = \allowbreak eX = \allowbreak Xe$. Therefore $X = eXe \in \{e\} \mathcal P(S) \{e\}$, and we are done. 
\end{proof}

Before proceeding, we work out an example demonstrating that
the inclusion in Eq.~\eqref{lem:localizers-in-power-semigroups-of-semilattices:eq(0)} cannot, in general, be replaced by equality. 

\begin{example}
Let $S = \{0,e,f\}$ be the three-element semilattice determined by the Cayley table
\[
\begin{array}{c|ccc}
\cdot & 0 & e & f \\ \hline
0     & 0 & 0 & 0 \\
e     & 0 & e & 0 \\
f     & 0 & 0 & f
\end{array}
\]

Suppose, by way of contradiction, that there exists $A \in \mathcal P(S)$ such that $
\Loc_{\mathcal P(S)}(A) = \{e\}\mathcal P(S)$. 
Since $\{e\} \in \{e\} \mathcal P(S)$, it is then clear that $Ae = \{e\}$. However, $e$ is the only element $a \in S$ with the property that $ae = e$. It follows that
$A = \allowbreak \{e\}$. On the other hand, direct verification gives
\[
\Loc_{\mathcal P(S)}(\{e\}) = \bigl\{\{e\}\bigr\}
\quad\text{and}\quad
\{e\}\mathcal P(S) \{e\} = \{e\}\mathcal P(S) = \mathcal P(\{0,e\})
 =\bigl\{\{0\},\{e\},\{0,e\}\bigr\},
\]
which is a contradiction. Therefore, $
\Loc_{\mathcal P(S)}(A)\ne\{e\}\mathcal P(S)$ for all $A\in\mathcal P(S)$.
\end{example}

Next we record an observation that, while elementary, will be essential in the subsequent proofs.

\begin{remark}
\label{lem:singleton-from-sandwich-product}
Let $S$ be a semigroup, and suppose that $ABA = A$ for some $A, B \in \mathcal P(S)$. If both $AB$ and $BA$ are singletons, then picking $a \in A$ and $b \in B$ yields $Ab = \{ab\}$ and $BA = \{ba\}$. Therefore, $A$ is itself a singleton, as we find that $A = A(BA) = A\{ba\} = (Ab)a = \{aba\}$.
\end{remark}

More precisely, we use this observation below to show that a homomorphism between power semigroups
that maps one-element idempotents to one-element idempotents
restricts to a homomorphism between the base semigroups (in the sense that it sends singletons to singletons),
provided that the underlying semigroup of the domain is regular. We recall that a semigroup $S$ is regular if,
for every $a \in S$, there is at least one element $b \in S$ such that $aba = a$.

\begin{proposition}
\label{prop:singleton-idempotents-and-regularity}
Let $H$ be a regular semigroup, $\mathcal D$ be a subsemigroup of $\mathcal P(H)$ containing every one-element subset of $H$, and $K$ be an arbitrary semigroup. If $\Phi$ is a hom\-o\-mor\-phism $\mathcal D \to \mathcal P(K)$ such that $\Phi(\{e\})$ is a singleton for every $e \in \EE(H)$, then there exists a homomorphism $g \colon H \to K$ such that 
\begin{equation}
\label{prop:singleton-idempotents-and-regularity:eq(0)}
\Phi(\{x\}) = \{g(x)\},
\qquad \text{for all } x \in H.
\end{equation}
\end{proposition}

\begin{proof}
Let $a \in H$. We first prove that $A:=\Phi(\{a\})$ is a
singleton. Since $H$ is regular, there exists $b \in H$ such
that $a=aba$. Setting $B:=\Phi(\{b\})$, we find that
\begin{equation}
\label{prop:singleton-idempotents-and-regularity:eq(1)}
ABA = \Phi(\{a\}) \Phi(\{b\}) \Phi(\{a\}) = \Phi(\{aba\}) = \Phi(\{a\}) = A,
\end{equation}
and in a similar fashion we obtain 
\begin{equation*}
\label{prop:singleton-idempotents-and-regularity:eq(2)}
AB = \Phi(\{ab\})
\quad\text{and}\quad
BA = \Phi(\{ba\}).
\end{equation*}

On the other hand, we have $(ab)^2 = (aba)b = ab$ and likewise $(ba)^2 = ba$, so $ab$ and $ba$ are idempotents of $H$. Thus, by hypothesis and the last display, $AB$ and $BA$ are singletons. Remark~\ref{lem:singleton-from-sandwich-product}, together with Eq.~\eqref{prop:singleton-idempotents-and-regularity:eq(1)}, therefore implies that $A$ is also a singleton, as desired.

It follows that there exists a function $g \colon H \to K$ such that Eq.~\eqref{prop:singleton-idempotents-and-regularity:eq(0)} holds. Since $\Phi$ is a homomorphism, it is then immediate to check that so is $g$.
\end{proof}

We are finally in a position to prove the main result of the section (Corollary \ref{cor:kobayashi-theorem}) and give an affirmative solution to Problem~\ref{ques:kobayashi-problem} for the class of inverse semigroups.

\begin{theorem}
\label{thm:kobayashi-theorem}
Let $H$ be an inverse semigroup, $K$ be an arbitrary semigroup, and $\Phi$ be an isomorphism from $\mathcal P(H)$ to $\mathcal P(K)$. There then exists an injective homomorphism $g \colon H \to K$ such that
\[
\Phi(\{x\}) = \{ g(x) \},
\qquad\text{for all }
x \in H.
\]
\end{theorem}

\begin{proof}
Let $e_H \in \EE(H)$. By Lemma \ref{lem:localizers-in-power-semigroups-of-semilattices}, there is an idempotent $A \in \mathcal E(\mathcal P(H))$ such that
\[
\{e_H\} \in \Loc_{\mathcal P(H)}(A) \subseteq \{e_H\} \mathcal P(H) \{e_H\}.
\]
Set $E := \Phi(\{e_H\})$ and $B := \Phi(A)$.  We have from the last display and Proposition \ref{prop:isos-preserve-localizers} that
\begin{equation}
\label{thm:kobayashi-theorem:eq(0)}
E \in \Loc_{\mathcal P(K)}(B) \subseteq E \mathcal P(K) E.
\end{equation}
Moreover, homomorphisms preserve idempotency. So, each of $B$ and $E$ is an idempotent of $\mathcal P(K)$, and we conclude from Eq.~\eqref{thm:kobayashi-theorem:eq(0)} and Corollary~\ref{cor:key-property-of-localizers} that $E = \{e_K\}$ for some $e_K \in \EE(K)$. 

It follows that $\Phi$ sends one-element idempotents of $\mathcal P(H)$
to one-element idempotents of $\mathcal P(K)$.
Since $H$ is regular,
Proposition~\ref{prop:singleton-idempotents-and-regularity}
then yields a homomorphism $g \colon H \to K$
such that $\Phi(\{x\}) = \{g(x)\}$ for every $x \in H$, and
the injectivity of $g$ is straightforward from that of $\Phi$.
\end{proof}

\begin{corollary}
\label{cor:kobayashi-theorem}
If $H$ and $K$ are inverse semigroups, then every isomorphism from $\mathcal P(H)$ to $\mathcal P(K)$ maps the one-element subsets of $H$ bijectively onto the one-element subsets of $K$.
\end{corollary}

\begin{proof}
By Theorem~\ref{thm:kobayashi-theorem}, there exists an injective homomorphism $g \colon H \to K$ such that
\[
\Phi(\{x\}) = \{g(x)\},
\qquad\text{for every } x \in H.
\]
Since $K$ is also an inverse semigroup and the functional inverse $\Phi^{-1}$ of $\Phi$ is an isomorphism from $\mathcal P(K)$ to $\mathcal P(H)$, another application of Theorem~\ref{thm:kobayashi-theorem} yields an
injective homomorphism $f \colon K \to H$ such that
\[
\Phi^{-1}(\{y\}) = \{f(y)\},
\qquad\text{for all } y \in K.
\]
Consequently, we obtain
\[
\{y\}
 = \Phi \circ \Phi^{-1}(\{y\})
 = \Phi(\{f(y)\})
 = \{g \circ f(y)\},
\qquad\text{for each } y \in K.
\]
This proves that $g \circ f$ is the identity map on $K$, and hence $g$ is also surjective. It follows that $\Phi$ maps the one-element subsets of $H$ bijectively onto the one-element
subsets of $K$.
\end{proof}

\begin{corollary}
Two inverse semigroups are globally isomorphic (that is, their large power semigroups are isomorphic) if and only if they are isomorphic.
\end{corollary}

\begin{proof}
The forward implication is immediate from
Corollary~\ref{cor:kobayashi-theorem}. For the converse, the augmentation of any semigroup isomorphism is an isomorphism between the corresponding large power semigroups, as already noted in Section~\ref{sec:intro}.
\end{proof} 
\section{Kobayashi's problem for inverse semigroups: the finitary case}
\label{sec:kobayashi-finitary}

In this section, we answer
Problem~\ref{ques:finitary-problems}\ref{ques:finitary-problems(ii)} in the affirmative for the class of inverse
semigroups (see Corollary~\ref{cor:finitary-kobayashi-for-inverse-semigroups}). We begin with a lemma with a vague topological flavor.

\begin{lemma}
\label{lem:finite-separation}
Let $S$ be an inverse semigroup, and fix an idempotent $e \in \EE(S)$ and a finite subfamily $\Gamma$ of
$\mathcal P_\fin(S)$ with $\{e\} \in \Gamma$. There then exists a non-empty finite set $A \subseteq \EE(S)$ such that
\[
e \in A = A^2
\quad\text{and}\quad
\{e\} \in \lLoc_{\mathcal P_\fin(S)}(A, \Gamma).
\]
\end{lemma}

\begin{proof}
Fix a set $X \in \Gamma$. If $X = \{e\}$, put $p_X := e$. Otherwise, $X \setminus \{e\}$ is non-empty, so pick $a_X \in \allowbreak X$ with $a_X \ne e$, and let $p_X$ be equal to $e$ or to $a_X a_X^{-1}$ according as $ea_X \ne e$ or $ea_X = e$. 

As already mentioned in the proof of Lemma~\ref{lem:localizers-in-power-semigroups-of-semilattices}, the product of an element $a \in S$ by its inverse $a^{-1}$ is idempotent
\cite[Theorem~2.3.4(1)]{Ho95}, while  \cite[Proposition~5.1.2(1)]{Ho95} yields $(ab)^{-1} = b^{-1}a^{-1}$ for every $a, b \in S$. Moreover, every idempotent is its own inverse in $S$. We claim that
\begin{equation}
\label{lem:finite-separation:eq(1)}
p_X \in \EE(S)
\quad\text{and}\quad
p_X e = e.
\end{equation}

This is obvious if $p_X = e$. Otherwise, $p_X = a_X a_X^{-1}$ and $e a_X = e$, so we conclude from the above remarks that $p_X$ is idempotent and, since $uv = vu$ for all $u, v \in \EE(S)$,
\[
e = e^{-1} = (ea_X)(ea_X)^{-1}
  = ea_X a_X^{-1}e^{-1} = ep_Xe = p_X e^2 = p_X e.
\]

Let $A$ be the subsemigroup of $S$ generated by the set 
\[
E := \{p_X : X \in \Gamma\} \subseteq S.
\]
Since $E$ is a finite subset of $\EE(S)$ and any two idempotents of $S$ commute (as noted earlier), every $a \in A$ is an idempotent and a finite product of pairwise distinct elements of $E$. By Eq.~\eqref{lem:finite-separation:eq(1)}, this implies that $ae = e$ and hence $Ae = \{e\}$. Moreover, $A$ is
a non-empty finite subset of $\EE(S)$ with $e \in A$ and
$A^2 = A$ (recall that $\{e\} \in \Gamma$ and hence $e = p_{\{e\}} \in E$).

Finally, if $X \in \Gamma$ and $X \ne \{e\}$, then
$p_X a_X \in AX$. However, $p_X a_X \ne e$, and hence
$AX \ne \{e\}$. Indeed, if $e a_X \ne e$, then $p_X = e$, so
$p_X a_X = e a_X \ne e$; otherwise, $p_X = a_X a_X^{-1}$, and so 
\[
p_X a_X = a_X a_X^{-1} a_X = a_X \ne e.
\]
Together with $Ae = \{e\}$ and $\{e\} \in \Gamma$, this suffices to show that
$\{e\} \in \lLoc_{\mathcal P_\fin(S)}(A, \Gamma)$.
\end{proof}

\begin{theorem}
\label{thm:finitary-kobayashi-theorem}
Let $H$ be an inverse semigroup, $K$ be an arbitrary semigroup, and $\Phi$ be an isomorphism from $\mathcal P_\fin(H)$ to $\mathcal P_\fin(K)$. There then exists an injective homomorphism $f \colon H \to K$ such that
\begin{equation}
\label{thm:finitary-kobayashi-theorem:eq(1)}
\Phi(\{x\}) = \{ f(x) \},
\qquad\text{for all }
x \in H.
\end{equation}
\end{theorem}

\begin{proof}
Fix $e \in \EE(H)$. The set $E := \Phi(\{e\})$ is an idempotent of $\mathcal P_\fin(K)$, because homomorphisms are idempotence-preserving functions. Our goal is to show that $E$ is also a singleton. Put 
\[
X_a := \Phi^{-1}(\{a\}) \in \mathcal P_\fin(H),
\qquad\text{for each }
a \in E. 
\]
Since a finite union of finite sets is finite, it is clear that
\[
W := \bigcup_{a \in E}\ \bigcup_{\emptyset \ne Y \subseteq X_a} \Phi(Y) \subseteq K
\]
is an element of $\mathcal P_\fin(K)$. In addition, we have
\begin{equation}
\label{thm:finitary-kobayashi-theorem:eq(inclusion)}
E = \bigcup_{a \in E} \{a\} = \bigcup_{a \in E} \Phi(X_a) \subseteq W.
\end{equation}

For brevity's sake, denote by $\mathcal Q(V)$ the family of all non-empty subsets of a set $V$. We have from the above that $\Gamma := \Phi^{-1}[\mathcal Q(W)]$ is a finite subfamily of $\mathcal P_\fin(H)$ containing $\{e\}$. Thus Lemma~\ref{lem:finite-separation} (applied with $S = H$) guarantees the existence of a non-empty finite set $A \subseteq \EE(H)$ such that
\begin{equation}
\label{thm:finitary-kobayashi-theorem:eq(2)}
e \in A = A^2 
\quad\text{and}\quad
\{e\} \in \lLoc_{\mathcal P_\fin(H)}(A, \Gamma).
\end{equation}
Taking $B := \Phi(A)$, it follows by Eq.~\eqref{thm:finitary-kobayashi-theorem:eq(2)} and Proposition \ref{prop:isos-preserve-localizers} (applied with $S = \mathcal P_\fin(H)$, $T = \mathcal P_\fin(K)$, and $\vartheta(X) := \Gamma$ for every non-empty finite $X \subseteq H$) that
\begin{equation}
\label{thm:finitary-kobayashi-theorem:eq(3)}
E \in \lLoc_{\mathcal P_\fin(K)}(B, \mathcal Q(W)).
\end{equation}

Next, let $\omega$ be the map
\[
\mathcal P_\fin(K) \to 2^{\mathcal P_\fin(K)} \colon Y \mapsto \mathcal Q(W \cup Y). 
\]
Each value of $\omega$ is a hereditary, union-closed family, and it is easy to check that $\omega$ is monotone under inclusion. Moreover, every $Y \in \mathcal P_\fin(K)$ is also an element of $\omega(Y)$, and if $Y \subseteq W$ then $\omega(Y) = \allowbreak \mathcal Q(W)$. Hence Proposition~\ref{lem:localizer-and-inclusion-hereditary} and the first display in Remark~\ref{remark:two-gauge-functions} (both applied with $S = K$, $\mathcal D = \allowbreak \mathcal P_\fin(K)$, and $\vartheta = \omega$), along with Eqs.~\eqref{thm:finitary-kobayashi-theorem:eq(inclusion)} and \eqref{thm:finitary-kobayashi-theorem:eq(3)}, imply that
\begin{equation}
\label{thm:finitary-kobayashi-theorem:eq(4)}
\{a\} \in \lLoc_{\mathcal P_\fin(K)}(B, \mathcal Q(W)),
\qquad\text{for all }
a \in E.
\end{equation}

Now fix $a \in E$, set $\Delta_a := \Phi[\mathcal Q(X_a)]$, and let $\delta_a$ be the constant map
\[
\mathcal P_\fin(K) \to 2^{\mathcal P_\fin(K)}: Y \mapsto \Delta_a.
\]
Since $\Phi(X_a) = \{a\}$, the definition of $W$ guarantees that 
$\{a\} \in \allowbreak \Delta_a \subseteq \mathcal Q(W)$.
Consequently, we get from Eq.~\eqref{thm:finitary-kobayashi-theorem:eq(4)} and the second display in Remark~\ref{remark:two-gauge-functions} (applied with $S = \mathcal P_\fin(K)$ and $\vartheta = \delta_a$) that
\[
\{a\} \in \lLoc_{\mathcal P_\fin(K)}(B, \Delta_a),
\]
which, by Proposition \ref{prop:isos-preserve-localizers} (applied with $S = \mathcal P_\fin(K)$, $T = \mathcal P_\fin(H)$, $\phi = \Phi^{-1}$, and $\vartheta = \delta_a$), 
yields
\[
X_a \in \lLoc_{\mathcal P_\fin(H)}(A, \mathcal Q(X_a)).
\]
The family $\mathcal Q(X_a)$ is hereditary. So, applying Proposition~\ref{prop:localizer-and-inclusion} (with $S = H$, $\mathcal D = \mathcal P_\fin(H)$, and $\vartheta(X) := \mathcal Q(X_a)$ for every $X \in \mathcal D$) and recalling from Eq.~\eqref{thm:finitary-kobayashi-theorem:eq(2)} that $e \in A$, we obtain $
\{e\}X_a = X_a$, whence 
\[
E\{a\} = \Phi(\{e\} X_a) = \Phi(X_a) = \{a\}.
\]

In fact, the preceding arguments and all the results used in them
(from the propositions and remarks of Section~\ref{sec:localizers}
to Lemma~\ref{lem:finite-separation}) admit left-right duals.
By symmetry, we therefore obtain
$\{a\} = E\{a\} = \{a\}E$.
Since $a \in E$ was arbitrary, it follows that
$x = xy = yx = y$ for all $x, y \in E$.
Thus, $E$ is a singleton.

To sum up, we have established that $\Phi$ maps one-element idempotents to one-element idempotents. Thus Proposition~\ref{prop:singleton-idempotents-and-regularity}, together with the regularity of $H$, gives a homomorphism $f \colon H \to K$ satisfying Eq.~\eqref{thm:finitary-kobayashi-theorem:eq(1)}, and the injectivity of $f$ is a direct consequence of that of $\Phi$ (cf.~the proof of Theorem \ref{thm:kobayashi-theorem}).
\end{proof}

\begin{corollary}
\label{cor:finitary-kobayashi-for-inverse-semigroups}
If $H$ and $K$ are inverse semigroups, then every isomorphism from $\mathcal P_\fin(H)$ to $\mathcal P_\fin(K)$ maps the one-element subsets of $H$ bijectively onto the one-element subsets of $K$.
\end{corollary}

\begin{proof}
Formally identical to the proof of  Corollary~\ref{cor:kobayashi-theorem}, with Theorem~\ref{thm:kobayashi-theorem} replaced by
Theorem~\ref{thm:finitary-kobayashi-theorem}.
\end{proof}

\begin{corollary}
\label{cor:finitary-tamura-shafer-for-inverse-semigroups}
The finitary power semigroups of two inverse semigroups are isomorphic if and only if the semigroups themselves are isomorphic.
\end{corollary}

\begin{proof}
The forward implication is straightforward from
Corollary~\ref{cor:finitary-kobayashi-for-inverse-semigroups}. For the converse, every semigroup
isomorphism $f \colon H \to K$ induces, through its augmentation, an isomorphism between the
corresponding finitary power semigroups, as already mentioned in
Section~\ref{sec:intro}.
\end{proof} 

\section*{Acknowledgments}

The authors were supported by the 111 Center under Grant No.~D26018 and by the Natural Science Foundation of Hebei Province under Grant No.~A2023205045.

\end{document}